\documentclass[12pt]{amsart}
\usepackage{amssymb, amsmath, amsfonts}
\usepackage[raiselinks=false,colorlinks=true,citecolor=blue,urlcolor=blue,linkcolor=blue,bookmarksopen=true,pdftex]{hyperref}
\newtheorem{theorem}{Theorem}[section]
\newtheorem{proposition}[theorem]{Proposition}
\newtheorem{lemma}[theorem]{Lemma}
\newtheorem{remark}[theorem]{Remark}

\newcommand{\bz}{\mathbb{Z}}
\newcommand{\bq}{\mathbb{Q}}
\newcommand{\br}{\mathbb{R}}
\newcommand{\bc}{\mathbb{C}}

\newcommand{\lr}{\longrightarrow}

\newcommand{\wt}{\widetilde}

\newcommand{\tr}{\textrm{tr}}

\newcommand{\gl}{\textrm{GL}}
\newcommand{\SL}{\textrm{SL}}

\begin{document}
\baselineskip=15.5pt
\title[Twisted conjugacy classes]{Twisted conjugacy classes in lattices in semisimple Lie groups}  
\author{T. Mubeena}
\author{P. Sankaran}
\address{The Institute of Mathematical Sciences, CIT
Campus, Taramani, Chennai 600113, India.}
\email{mubeena@imsc.res.in}
\email{sankaran@imsc.res.in}

\subjclass{20E45, 22E40, 20E36\\
Key words and phrases: Twisted conjugacy classes,  lattices in semisimple Lie groups, Groups with $R_\infty$ property}

\date{}

\begin{abstract}
 Given a group automorphism $\phi:\Gamma\lr \Gamma$,  one has 
an action of $\Gamma$ on itself by $\phi$-twisted conjugacy, namely, $g.x=gx\phi(g^{-1})$. 
The orbits of this action are called $\phi$-conjugacy classes.  One says 
that $\Gamma$ has the $R_\infty$-property if there are infinitely many $\phi$-conjugacy 
classes for every automorphism $\phi$ of $\Gamma$. In this paper we show that any irreducible lattice in 
a connected non-compact semi simple Lie group having finite centre and rank at least $2$ has the $R_\infty$-property. 
\end{abstract}
\maketitle
%%%%%%%%%%%%%%%%%%%%%%%%%%%%
\section{Introduction}
Let $\Gamma$ be a finitely generated infinite group and let $\phi:\Gamma\lr \Gamma$ be an endomorphism.  
One has an equivalence relation $\sim_\phi$ on $\Gamma$ defined as $x\sim_\phi y$ if there exists a $g\in \Gamma$ such that $y=gx\phi(g)^{-1}$.  The equivalence 
classes are called the $\phi$-conjugacy classes. 
Note that when $\phi$ is the identity, $\phi$-conjugacy classes are the usual conjugacy classes. 
The $\phi$-conjugacy classes are nothing but the orbits of the action of $\Gamma$ on itself defined as $g.x=gx\phi(g^{-1})$.   The $\phi$-conjugacy class containing $x\in \Gamma$ is denoted $[x]_\phi$ or simply $[x]$ when $\phi$ is clear from the context.  The set of all $\phi$-twisted conjugacy classes is denoted by $\mathcal{R}(\phi)$. The cardinality $R(\phi)$ of $\mathcal{R}(\phi)$ is called the {\it Reidemeister number} of 
$\phi$.  One says that $\Gamma$ has the $R_\infty$-property for automorphisms (more briefly, $R_\infty$-property) if there are infinitely many $\phi$-twisted conjugacy classes for every automorphism $\phi$ of $\Gamma$.  If $\Gamma$ has the $R_\infty$-property, we shall call $\Gamma$ an $R_\infty$-group.  

The notion of twisted conjugacy originated in 
Nielson-Reidemeister fixed point theory and also arises in other areas of mathematics such as representation theory, number theory and algebraic geometry. See \cite{felshtyn2} and the references therein. 
The problem of determining which classes of groups have $R_\infty$-property is an area of active research initiated by Fel'shtyn and Hill \cite{fh}.  

Let $G$ be a non-compact semi simple Lie group with  finite centre.  Recall that a discrete subgroup 
$\Gamma\subset G$ is called a {\it lattice} if $G/\Gamma$ has a finite $G$-invariant measure. One says that $\Gamma$ is {\it cocompact} if $G/\Gamma$ is compact; otherwise $\Gamma$ is non-cocompact.  If, for any  non-compact closed normal subgroup $H\subset G$, the image of $\Gamma$ under the quotient map $G\lr G/H$ is 
dense, one says that $\Gamma$ is {irreducible.} 
If $G$ has no compact factors, 
$\Gamma $ is irreducible if and only if  
for any two closed normal subgroups $H_1, H_2$ of $G$ such that $G=H_1.H_2$ and lattices $\Gamma_i\subset H_i$, the 
group $\Gamma_1.\Gamma_2$ is not commensurable 
with $\Gamma$. 
In particular, any lattice 
in $G$ is irreducible if $G$ is simple.

The main result of this paper is the following:
%%%%%%%%%%%%%%%%%%%%%%%
\begin{theorem} \label{main}
Let $\Gamma$ be any irreducible lattice in a connected semi simple non-compact Lie group $G$ with finite centre. If the real rank of $G$ is at least $2$, then 
$\Gamma$ has the $R_\infty$ property.
\end{theorem} 

When $G$ has real rank $1$, the above result is 
well-known.  Indeed, assume that $G$ has real rank $1$. When the lattice  
$\Gamma$ is cocompact, it is hyperbolic. When 
$\Gamma$ is not cocompact, it is relatively hyperbolic.  
It has been shown by Levitt and Lustig \cite{ll} that 
any torsion free non-elementary hyperbolic group has the $R_\infty$-property. Fel'shtyn (\cite{felshtyn1},\cite{felshtyn2}) established the $R_\infty$ property 
for arbitrary non-elementary hyperbolic groups as well as 
non-elementary relatively hyperbolic groups. 

When $\Gamma$ is a principal congruence subgroup of $\SL(n,\bz), $ 
the above theorem was established in \cite{ms}.  When $\Gamma=Sp(2n,\bz)$, the result was first 
proved by Fel'shtyn and Gon\c{c}alves \cite{fg}; see also \cite{ms}. 
 
Our proof of the above theorem involves only elementary  arguments, using some well-known but deep results concerning irreducible lattices in semi simple Lie groups.
The main theorem is first established when $G$ has no compact factors and has trivial centre.  In this case, the proof uses the Zariski density property of $\Gamma$ due to Borel as well as the strong rigidity theorem. When 
$G$ has non-trivial compact factors, we need to 
use Margulis' normal subgroup theorem to 
reduce to the case when $G$ has trivial centre and 
no compact factors.   

In \S2 we shall recall the results on lattices in 
semi simple Lie groups needed in the proof of Theorem \ref{main}, given in \S3. 

%%%%%%%%%%%%%%%%%%%%%%%
\section{Lattices in semi simple Lie groups}

We recall below the definition of an arithmetic lattice in a semi simple Lie group and some deep results concerning them relevant for our purposes.    

Let ${\bf G}\subset \gl(n,\bc)$ be an algebraic group, that is, ${\bf G}$ is a subgroup of $\gl(n,\bc)$ such that $\mathbf{G}$ is the zero locus of a collection of (finitely many) polynomial equations $f_m(X_{ij})=0$ in the $n^2$ matrix entries $X_{i,j}, 1\leq i,j\leq n$.
One says that ${\bf G}$ is defined over a subfield $k\subset \bc$ if the $f_m$ can be chosen to 
have coefficients in $k$; in this case $G_k:={\bf G}\cap \gl(n,k)$ is the $k$-points of ${\bf G}$.  
If $R$ is a subring of $k$, then $G_R:={\bf G}\cap \gl(n,R)$.  
A theorem of Borel and Harish-Chandra asserts that if ${\bf G}$ is a connected semi simple algebraic group defined over $\bq$ then $G_\bz$ is a lattice in $G_\br$.  
We say that a lattice $\Gamma\subset G$ is {\it arithmetic} if $\mathbf{G}$ is defined over $\bq$ and if $\Gamma$ is commensurable with $G_\bz$.  

If $N\subset G$ is a compact normal subgroup of a connected Lie group $G$ with finite centre and $\Gamma$ a discrete subgroup of  $G$, 
then $\Gamma$ is a lattice in $G$ if and only if the image of $\Gamma$ under the quotient map 
$G\lr G/N$ is a lattice.

Let $G$ be a connected semi simple Lie group and 
$K\subset G$ a maximal compact subgroup. The 
{\it real rank} of $G$ is 
the largest integer $m$ such that the Euclidean space $\br^m$ can be imbedded as a totally geodesic submanifold of the symmetric space $G/K.$  Equivalently, the real rank of $G$ is the dimension of the 
largest abelian subalgebra contained in $\frak{p}$ 
where $\frak{g}=\frak{k}\oplus \frak{p}$ is the Cartan 
decomposition.  Here $\frak{g}=Lie(G), \frak{k}=Lie(K)$.  

The following well-known results will be needed in the proof of our main theorem. 

\begin{theorem}{\em (Borel density theorem)}
Let $\Gamma\subset G_\br$ be any lattice in a connected 
semi simple algebraic group $\mathbf{G}$ defined over $\bq$.  If $G_\br$ has no compact factors,  
then $\Gamma$ is Zariski dense in $\mathbf{G}$.   \hfill $\Box$
\end{theorem}

\begin{theorem} \label{normal}{\em (Margulis' normal subgroup theorem)}
Let $\Gamma\subset G$ be an irreducible lattice where $G$ is a 
connected semi simple Lie group of rank at least $2$ and with finite centre.  
If $N$ is normal in $\Gamma$, then either $N$ is of finite 
index in $\Gamma$ or is a finite subgroup contained in the centre of $G$. \hfill $\Box$
\end{theorem}

Next we state the strong rigidity for irreducible lattices.  

\begin{theorem} \label{rigidity} (Strong rigidity)
Let $G$ and $G'$ be connected linear semi simple Lie groups with trivial centre and having no compact factors. Let $\Gamma\subset G$ and $\Gamma'\subset G'$ be irreducible lattices.  Assume that $G$ and $G'$ are not locally isomorphic to $SL(2,\br)$. Then any isomorphism $\phi:\Gamma\lr \Gamma'$ extends to an isomorphism $G\lr G'$ of Lie groups. \hfill $\Box$
\end{theorem}

The strong rigidity theorem for cocompact lattices was obtained by Mostow \cite{mostow}.  Margulis showed that the result holds for $G$ as above with real rank $\geq 2$.  The rank $1$ case (when the lattice is non-cocompact) is due to 
Prasad \cite{prasad}, who extended the classical work of Mostow concerning rigidity of 
rank $1$ compact locally symmetric manifolds. The proofs of the rigidity theorem for the case rank $\geq 2$, the Borel density theorem, and the Margulis' normal subgroup theorem can be found in \cite{zimmer}.

\section{Proof of Theorem \ref{main}}

Before we begin the proof, we recall some elementary 
notions in combinatorial group theory and recall some 
facts concerning the $R_\infty$-property.

Let $\Gamma$ be a group and $H$ a subgroup of $\Gamma$.  Recall that a subgroup $H$ is said to be {\it  characteristic} in $\Gamma$ if $\phi(H)=H$ for every automorphism $\phi$ of $\Gamma$.  $\Gamma$ is called {\it hopfian} (resp. {\it co-hopfian}) if every surjective (resp. injective) endomorphism of $\Gamma$ is an automorphism of $\Gamma$.   One says that $\Gamma$ is {\it residually finite}  
if, given any $g\in \Gamma$, there exists a finite index subgroup $H$ in $\Gamma$ such that $g\notin H$. It is well-known 
that any finitely generated subgroup of $\gl(n,k)$, where $k$ is any field,  
is residually finite and that finitely generated 
residually finite groups are hopfian.  We refer the reader to 
\cite{ls} for detailed discussion on these notions.

We recall here some facts concerning the $R_\infty$-property. Let 
\[1\lr N\stackrel{j}{\hookrightarrow} \Lambda\stackrel{\eta}{\lr}\Gamma\lr 1\eqno(1)\] 
be an exact sequence of groups.    
  
\begin{lemma} \label{characteristic} 
Suppose that $N$ is characteristic in $\Lambda$ and that $\Gamma$ has the $R_\infty$-property, then $\Lambda$ also has the $R_\infty$-property.
\end{lemma}

\begin{proof}
Let $\phi:\Lambda\lr \Lambda$ be any automorphism. Since $N$ is characteristic, $\phi(N)=N$ and so $\phi$ induces an automorphism $\bar{\phi}:\Gamma\lr \Gamma$. Since $R(\bar{\phi})=\infty$, it follows that $R(\phi)=\infty$.   
\end{proof}

The following proposition is perhaps well-known; a proof can be found in \cite{ms}.

\begin{proposition}\label{residual}
Let $\Gamma$ be a countably infinite residually finite group. Then $R(\phi)=\infty$ for any inner automorphism $\phi$ of $\Gamma$.  \hfill $\Box$
\end{proposition}

We are now ready to prove the main theorem.

\noindent
{\it Proof of  Theorem \ref{main}:}  
First suppose that $G$ has trivial centre and has no compact 
factors.   Since the centre of $G$ is trivial, the homomorphism $\iota:G\lr Aut(G)$ given by inner 
automorphism allows us to identify $G$ with the group of inner automorphims of $G$.  Under this identification, 
$G$ is the identity component of $Aut(G)$ and 
$Aut(G)/G\cong Out(G)$ is finite.  
Also the group $Aut(G)$ is isomorphic to the linear Lie group $Aut(\frak{g})\subset \gl(\frak{g})$ of the automorphisms of the Lie algebra $\frak{g}$ of $G$ under which $\phi\in Aut(G)$ corresponds to its derivative  at the identity element.  
Thus we have a chain of monomorphisms $\Gamma\hookrightarrow G\stackrel{\iota}{\lr} Aut(G)\cong Aut(\frak{g})\hookrightarrow \gl(\frak{g})$.   Furthermore,  $Aut(G)\cong Aut(\frak{g})$ is the $\br$-points of the complex algebraic group $\mathbf{H}:=Aut(\frak{g}\otimes_\br \bc)$ and the identity component of $H_\br$ is $Aut(G)^0=G$.  

Suppose that $\phi:\Gamma\lr \Gamma$ is an automorphism. 
Clearly $\phi\circ \iota_\gamma=\iota_{\phi(\gamma)}\circ\phi$ where $\iota_\gamma$ denotes 
conjugation by $\gamma$. 
Now let $x,y\in \Gamma$ be such that $x\sim_\phi y$. 
Then there exists a $\gamma\in \Gamma$ such that $y=\gamma x\phi(\gamma^{-1})$; equivalently,  
$\iota_y=\iota_\gamma \iota_x\iota_{\phi(\gamma)^{-1}} =\iota_\gamma \iota_x\phi\iota_{\gamma^{-1}}\phi^{-1}$.  Hence
$\iota_y\phi=\iota_\gamma(\iota_x\phi)\iota_{\gamma^{-1}}.$ 
  
By the strong rigidity theorem, 
$\phi\in Aut(\Gamma)$ extends to an automorphism of the Lie group $G$, again denoted $\phi\in Aut(G)$.
For any $h\in H_\br$, consider the function $\tau_h:\mathbf{H}\lr \bc$ defined as $\tau_h(x)=\tr(xh)$, the trace of 
$xh\in \mathbf{H}\subset \gl(\frak{g}\otimes_\br\bc)$.  
Clearly this is a morphism of varieties defined over $\br$.  We have that, if $x,y\in \Gamma, x\sim_\phi y$, then $\tau_\phi(y)=\tau_\phi(x)$ since $\iota_y\phi$ and $\iota_x\phi$ are conjugates in $\mathbf{H}$.   

Assume that the Reidemeister number of $\phi$ is finite. 
Then, by what has been observed above, $\tau_\phi$ assumes only finitely many values on $\Gamma
\subset H_\br^0=G$. 
Since, by the Borel density theorem, $\Gamma$ is Zariski dense in $\mathbf{H}^0$, it follows that $\tau_\phi$ is constant on $\mathbf{H}^0$.   This clearly implies that 
$\tau_{h\phi }$ is constant for {\it any} $h\in H^0_\br$.

Let $K$ be a maximal compact subgroup of $H_\br=Aut(G)$. Since $Aut(G)$ has only finitely many components, by a well-known result of Mostow, $K$ meets {\it every} connected component of $Aut(G)$.  (See \cite[Theorem 1.2, Ch. VII]{bor},\cite{ho}.)  Thus $K$ contains representatives of every element of $Out(\Gamma)$ and so we may choose an $h\in H^0_\br$ such that $\theta:=h\phi \in K$. 
The automorphism $Ad(\theta)$ on the Lie algebra $Lie(K^0)$ fixes a regular (semi simple) element $X\in Lie(K^0)$ by \S 3.2, Ch. VII of \cite{bor}. Hence the one-parameter subgroup $S:=\{\exp(tX)\mid t\in \br\} \subset K^0$ is contained in the centralizer $C_{H_\br}(\theta)
=\{x\in H_\br\mid \theta x=x \theta\}$.  
Note that $\theta$ is also semi simple since $K$ is compact subgroup of $\gl(\frak{g}\otimes_\br\bc)$. 
It follows that $\theta$ and $\exp(tX), t\in \br,$ are simultaneously diagonalizable (over $\bc$).  It is now 
readily seen that $\tau_\theta$ is not constant on 
$S\subset H_\br^0 $, a contradiction to our earlier 
observation that $\tau_{h\phi}$ is a constant function for any $h\in H^0_\br$. This implies that $R(\phi)=\infty$. 

Next suppose that $G$ has no compact factors but  possibly with non-trivial centre, $Z$. By our hypothesis $Z$ is finite. Clearly $Z\cap \Gamma\subset Z(\Gamma)$ the centre of $\Gamma$. 
Since $\bar{\Gamma}:=\Gamma/(Z\cap \Gamma)$ is Zariski dense in $G/Z$, and since $G/Z$ has trivial centre, we see that $\Gamma/(Z\cap \Gamma)$ has trivial centre. It follows that  $Z(\Gamma)= Z\cap \Gamma$.  
Consider the exact sequence 
\[1\to Z\cap \Gamma\to \Gamma\to \bar{\Gamma}\to 1\eqno{(2)}\]
Since $Z\cap\Gamma=Z(\Gamma)$ is a finite characteristic subgroup of $\Gamma$,  the $R_\infty$ property for $\Gamma$ follows from that for $\bar{\Gamma}$.  

Finally let $G$ be any Lie group as in the theorem. 
Let $M$ be the maximal compact normal subgroup of $G$.  Note that 
	$M$ contains the centre $Z$ of $G$.  Now $M\cap \Gamma$ is a {\it finite} normal subgroup of $\Gamma$.   
We invoke Theorem \ref{normal} to conclude that 
$M\cap \Gamma$ is contained in the centre of $G$.   Also $Z(\Gamma)$ is contained in $Z$ since, otherwise, by Theorem \ref{normal} again, $\Gamma$ would be virtually abelian. Since $G$ is a non-compact semi simple Lie group, this is impossible.  
Since $M$ contains 
$Z$, we see that $M\cap \Gamma=Z\cap \Gamma$ {\it equals} 
the centre of $\Gamma$ and hence is characteristic  
in $\Gamma$. 
Now $\bar{\Gamma}:=\Gamma/(M\cap\Gamma)$ is an irreducible lattice in $G/M$, which has trivial centre and no compact factors.  Using the exact sequence (2) again, we see that $R(\phi)=\infty$. 
This completes the proof. \hfill $\Box$

\begin{remark} {\em 
(i)  Suppose that $G$ is not locally isomorphic to $SL(2,\br)$ and that the real rank of $G$ equals $1$.  When $G$ has no compact factors, the above proof 
can be repeated verbatim to show that $\Gamma$ has the $R_\infty$ property.  When $G$ has 
compact factors and $\Gamma$ is residually finite (for example when $G$ is linear) one can find 
a finite index characteristic subgroup $\Gamma'$ of $\Gamma$ such that $\Gamma'\cap M=\{1\}$ where $M$ is as in the above proof. Now $\Gamma'\cong \Gamma'/M$ and so has the $R_\infty$ property. 
It follows from Lemma 2.2 of \cite{ms} that $\Gamma$ has the $R_\infty$ property.\\
(ii) Suppose that $G$ is a linear connected 
semi simple Lie group of real rank at least $2$ and 
let $\Gamma$ be an irreducible lattice in $G$. Since 
$\Gamma$ is finitely generated and linear, it follows that 
$\Gamma$ is residually finite and hence Hopfian.  
Let $1\to A\stackrel{j}{\hookrightarrow} \Lambda\stackrel{\eta}{\to} \Gamma\to 1$ be an exact sequence of groups where $A$ is any countable abelian group.  Proceeding as in the proof of \cite[Theorem 1.1(ii)]{ms}, one can show that $\Lambda$ has the $R_\infty$-property. We give an outline of the proof.  Let $\phi\in Aut(\Lambda)$ and let $f=\eta\circ\phi|A$.  Then $f(A)$ is normal in $\Gamma$.  By the normal subgroup theorem 
of Margulis, either $f(A)$ is of finite index---in which case $f(A)$ is a lattice in $G$---or $f(A)$ 
is contained in the centre of $G$ since $G$ has real rank at least $2$ and $\Gamma$ is irreducible.  Since $\Gamma$ is not virtually abelian, we see that $f(A)$ has to be finite. Replacing $A$ by $\wt{A}:=\eta^{-1}(Z(\Gamma))$ we see that $\wt{A}$ is a characteristic subgroup of $\Lambda$.    Using the observation that $\Gamma$ is Hopfian and proceeding as in \cite{ms}, we see that $\Lambda$ has the $R_\infty$ property.    \\
(iii) Timur Nasibullov \cite{n} has obtained the following result.   Let $\Gamma=GL(n,R)$ or $SL(n,R), n\geq 3,$ where $R$ is an infinite integral domain and let $\Phi$ be the subgroup of $Aut(\Gamma)$ generated by the inner automorphisms, homothety by a central character,  and the contragradient automorophisms. Then for any $\phi\in \Phi,$ one has $R(\phi)=\infty$.  In particular, if $R$ has no non-trivial automorphism (e.g. $R=\br$) and has characteristic zero, then 
$\Gamma$ has the $R_\infty$-property.  
}
\end{remark}

%%%%%%%%%%%%%%%%%%%%%%%%%%%%%%%%%%%%%%


\begin{thebibliography}{99}
\bibitem{bor} A. Borel, {\it Semisimple groups and Riemannian symmetric spaces}, Texts and Readings in Mathematics {\bf 16}, Hindustan Book Agency, New Delhi, 1998.  
\bibitem{felshtyn1} A. Fel'shtyn, 
{\it The Reidemeister number of any automorphism of a Gromov hyperbolic group is infinite}. Zap. Nauchn. Sem. S.-Peterburg. Otdel. Mat. Inst. Steklov. (POMI) {\bf 279} (2001), Geom. i Topol. {\bf 6}, 229--240.  
\bibitem{felshtyn2} A. Fel'shtyn,  New directions in Nielsen--Reidemeister theory.  Topology Appl. {\bf 157} (2010) 1724--1735.
\bibitem{fg} A. Fel'shtyn and D. L. Gon\c{c}alves, {\it Twisted conjugacy classes in symplectic groups, mapping class groups and braid groups}. With an appendix written jointly with Francois Dahmani. Geom. Dedicata {\bf 146} (2010), 211--223.
\bibitem{fh} A. Fel'shtyn and R. Hill, 
{\it The Reidemeister zeta function with applications to Nielsen theory and a connection with Reidemeister torsion}, K-Theory {\bf 8} (1994) 367--393.

\bibitem{ho} G. Hochschild, {\it The structure of Lie groups}, Holden-Day, 1965.
\bibitem{ll} G. Levitt and M. Lustig,  {\it Most  automorphisms of a hyperbolic group have very simple dynamics}. Ann. Sci. Ecol. Norm. Sup. {\bf 33} (2000) 507--517.
\bibitem{ls} R. Lyndon and P. Schupp,
{\it Combinatorial group theory}. Reprint of the 1977 edition. Classics in Mathematics. Springer-Verlag, Berlin, 2001.
\bibitem{mostow} G. D. Mostow, {\it Strong rigidity of locally symmetric spaces.} Annals of Mathematics Studies,  {\bf 78}. Princeton University Press, Princeton, N.J.; University of Tokyo Press, Tokyo, 1973.
\bibitem{ms} T. Mubeena and P. Sankaran, 
{\it Twisted conjugacy classes in abelian extensions of 
certain linear groups}, preprint, 2011. 
\bibitem{n} Timur Nasibullov, {\it Twisted conjugacy classes in general and special linear groups} (in Russian), 21/Oct/2011, preprint.
\bibitem{prasad}  G. Prasad, {\it Strong rigidity of $\bq$-rank $1$ lattices}, Invent. Math {\bf 21} 255--286. 
\bibitem{zimmer} R. J. Zimmer, {\it Ergodic theory and semisimple Lie groups}, Monographs in
Mathematics {\bf 81}, Birkh\"{a}user, 1984.
\end{thebibliography}
\end{document}